\documentclass[11pt,reqno]{amsart}
\usepackage{amsmath,amsopn,amssymb,amsthm,multicol}
\usepackage{hyperref}
\usepackage[all]{xy}
\usepackage{graphicx}
\usepackage{graphics}

\usepackage{epstopdf}

\DeclareMathOperator{\Ad}{Ad}

\DeclareMathOperator{\Aut}{Aut}

\DeclareMathOperator{\tr}{tr}

\DeclareMathOperator{\Ric}{Ric}

\newcommand{\fr}{\mathfrak}
\newcommand{\al}{\alpha}

\newcommand{\bb}{\mathbb}

\DeclareMathOperator{\SO}{SO}

\DeclareMathOperator{\SU}{SU}
\DeclareMathOperator{\U}{U}

 \newtheorem{lemma} {Lemma} [section]
\newtheorem{theorem}[lemma]{Theorem} 
\newtheorem{remark}[lemma] {Remark} 
\newtheorem{prop} [lemma]{Proposition}

\usepackage{array}
\makeatletter
\newcommand{\thickhline}{%
    \noalign {\ifnum 0=`}\fi \hrule height 1pt
    \futurelet \reserved@a \@xhline
}
\newcolumntype{"}{@{\hskip\tabcolsep\vrule width 1pt\hskip\tabcolsep}}
\makeatother

\begin{document}

\title{Homogeneous Einstein metrics on  Stiefel manifolds associated to flag manifolds with two isotropy summands} 
\author{Andreas Arvanitoyeorgos, Yusuke Sakane and Marina Statha}
\address{University of Patras, Department of Mathematics, GR-26500 Rion, Greece}
\email{arvanito@math.upatras.gr}
 \address{Osaka University, Department of Pure and Applied Mathematics, Graduate School of Information Science and Technology, Suita, Osaka 565-0871, Japan}
 \email{sakane@math.sci.osaka-u.ac.jp}
\address{University of Patras, Department of Mathematics, GR-26500 Rion, Greece}
\email{statha@master.math.upatras.gr} 

\medskip
 
\begin{abstract}
We study invariant Einstein metrics on the Stiefel manifold $V_{k}\bb{R}^{n}\cong \SO(n)/\SO(n-k)$ of all orthonormal $k$-frames in $\bb{R}^{n}$.  The isotropy representation of this homogeneous space contains equivalent summands, so a complete description of $G$-invariant metrics is not easy.  
In this paper we view the manifold $V_{2p}\bb{R}^{n}$ as total space over a classical generalized flag manifolds with two isotropy summands and prove for $2\le p\le \frac25 n-1$ it admits at least four invariant Einstein metrics determined by $\Ad(\U(p)\times\SO(n-2p))$-invariant scalar products.  Two of the metrics are Jensen's metrics and the other two are new Einstein metrics. 
The Einstein equation reduces to a  parametric system of polynomial equations, which we study  by
combining Gr\"obner bases and geometrical arguments.


\end{abstract}

\maketitle

\section{Introduction}
\markboth{A. Arvanitoyeorgos, Y. Sakane and M. Statha}{Homogeneous Einstein metrics on real Stiefel manifolds}

A Riemannian metric $g$ on a manifold $M$ is called Einstein if the Ricci tensor is a constant multiple of the metric, i.e. $\Ric_g = c\cdot g$ for some $c\in\bb{R}$.  In this case the Riemannian manifold $(M, g)$ is called Einstein.  The Einstein equation is a non linear second order system of partial differential equations and the general existence results are difficult to obtain.  We refer to \cite{Be} and \cite{W1}, \cite{W2} for old and new results on homogeneous Einstein manifolds.  The structure of the set of invariant Einstein metrics on a given homogeneous space is still not very well understood in general.  The situation is only clear for few classes of homogeneous spaces, such as isotropy irreducible homogeneous spaces, low dimensional examples, certain flag manifolds.  For an arbitrary compact homogeneous space $G/H$ it is not clear if the set of invariant Einstein metrics (up to isometry and up to scaling) is finite or not.  A finiteness conjecture states that this set is in fact finite if the isotropy representation of $G/H$ consists of pairwise inequivalent irreducible subrepresentations (\cite{BWZ}).  
 In \cite{ADN1} the authors introduced a method for proving existence of homogeneous Einstein metrics by assuming additional symmetries.  In \cite{St} a systematic and organized description of such metrics is presented.

In the present article we are interested for invariant Einstein metrics on homogeneous spaces $G/H$ whose isotropy representation is decomposed into a sum of irreducible but possibly equivalent summands.  Typical examples of such homogeneous spaces are the (real) Stiefel manifolds.
These are the homogeneous spaces $V_{k}\bb{R}^{n}\cong\SO(n)/\SO(n-k)$, which are diffeomorphic to the set of orthonormal $k$-frames in $\bb{R}^n$.  The simplest cases are the sphere $\bb{S}^{n-1}\cong\SO(n)/\SO(n-1)$ and the compact Lie group $V_{n}\bb{R}^{n}\cong \SO(n)$. 

Since a complete description of $G$-invariant metrics on such spaces (and in turn computation of Ricci tensor) is not easy, a helpful approach is the following (cf \cite{ADN1}).
We can choose a closed subgroup $K$ of $G$ such that $H\subset K\subset N_{G}(H)$ and search for $\Ad(K)$-invariant scalar products on $T_o(G/H)$, which correspond to  a subset of all
$G$-invariant metrics on $G/H$, called $\Ad(K)$-invariant metrics.  The benefit is that these metrics could be diagonal, hence search for Einstein metrics among them.
To this end we can consider appropriate fibrations of $G/H$ over some  homogeneous space $G/K$ whose isotropy representation is decomposed into non equivalent summands.

 Concerning history, in \cite{Ko} S. Kobayashi proved existence of an invariant Einstein metric on the unit tangent bundle $T_1S^n = \SO(n)/\SO(n-2)$.  In \cite{Sa} A. Sagle proved that the Stiefel manifolds $V_k\mathbb{R} ^n=\SO(n)/\SO(n-k)$ admit at least one homogeneous Einstein metric.  For $k\ge 3$ G. Jensen in \cite{J2} found a second metric.  If $n=3$ the Lie group $\SO(3)$ admits a unique Einstein metric.  For $n\ge 5$  A. Back and W.Y. Hsiang in \cite{BH} proved that $\SO(n)/\SO(n-2)$ admits exactly one homogeneous Einstein metric. The same result was obtained by M. Kerr in \cite{Ke} by proving that the diagonal metrics are the only invariant metrics on $V_2\mathbb{R}^n$ (see also \cite{A1}, \cite{AK}).  The Stiefel manifold $\SO(4)/\SO(2)$ admits exactly two invariant Einstein metrics (\cite{ADF}).  One is Jensen's metric and the other one is the product metric on $S^3\times S^2$.  

In \cite{ADN1} the first author, V.V. Dzhepko and Yu. G. Nikonorov proved that
for $s>1$ and $\ell >k\ge 3$ the Stiefel manifold $\SO(s k+\ell)/\SO(\ell)$ admits at least four $\SO(s k+\ell)\times (\SO(k))^s$-invariant Einstein metrics, two of which are Jensen's metrics.  

In \cite{ArSaSt1} the authors used fibrations of Stiefel manifolds over the generalized Wallach spaces $\SO(k_1+k_2+k_3)/\SO(k_3)$ and
proved the following:
 The Stiefel manifolds
$V_4\bb{R}^n=\SO(n)/\SO(n-4)$ $(n\ge 6)$ admit at least four invariant Einstein metrics.  Two are Jensen's metrics and the other two are new metrics.
The Stiefel manifold
$V_5\bb{R}^7=\SO(7)/\SO(2)$ admits at least six invariant Einstein metrics.  Two are Jensen's metrics and the other four are new metrics.
A generalization of this result for $V_5\bb{R}^n=\SO(n)/\SO(n-5)$ $(n\ge 7$) was given by the first author  in \cite{A2}.

For an overview about invariant Einstein metrics on quaternionic and complex Stiefel manifolds we refer to \cite{ArSaSt3} and 
\cite{ArSaSt2} respectively.

In the present paper we obtain new invariant Einstein metrics on the Stiefel manifolds (A) $V_{2 p}\bb{R}^{2\ell}$ and (B) $V_{2 p}\bb{R}^{2\ell+1}$.  In case (A) we view the Stiefel manifold as total space over the generalized flag manifold $D_{\ell}(p) = \SO(2\ell)/\U(p)\times\SO(2(\ell-p))$  and for the case (B) we take the generalized flag manifold $B_{\ell}(p) = \SO(2\ell + 1)/\U(p)\times\SO(2(\ell-p)+1)$.  The main result is the following:

\noindent{\bf Theorem.}
For $ \displaystyle 2 \leq p \leq \frac{2}{5} n-1$,   the Stiefel manifold $V_{2 p}\bb{R}^{n} = \SO(n)/\SO(n-2 p)$ admits at least four invariant Einstein metrics which are determined by  $\Ad(\U(p)\times\SO(n - 2 p) )$-invariant scalar products. Two  of these metrics are Jensen's metrics and the other two are different from Jensen's metrics.

The  idea of our approach is the following. 
The isotropy representation of the Stiefel manifolds $G/H=\SO(n)/\SO(n-2 p)$ contains equivalent subrepresentations, hence a description of all 
$\SO(n)$-invariant metrics is not easy (cf. Section 3).  These metrics correspond to $\Ad(H)$-invariant scalar products on the tangent space $\mathfrak{m}$ of $G/H$ at the identity coset $eH$.  We choose an appropriate
 closed subgroup $K$ of $G$ such that $H \subset K\subset N_G(H)$, the normalizer of $H$ in $G$, and search for 
$\SO(n)$-invariant metrics that correspond to 
 $\Ad(K)$-invariant scalar
products on $\mathfrak{m}$ (cf. Section 2).  The benefit of such metrics is that they are diagonal metrics
on the homogeneous $G/H$, hence computation of Ricci tensor simplifies (cf. Section 3).

Next, it is well known that the Einstein equation for a homogeneous space reduces to an algebraic system of equations.  In our case we obtain
 a parametric system of three polynomial equations
with three unknowns $u_0, u_1, u_2$ (cf.\,(\ref{system4})), and we need to prove existence of positive solutions.  Furthermore, we need to distinguish the solutions detected, from the solutions that correspond to Jensen's Einstein metrics.
In fact, these known solutions confirm that out computations are right.
The proof of this algebraic claim is based on an analysis of the parametric system of
algebraic equations based on computations with Gr\"obner bases, by making appropriate choice of lexicographic ordering for the variables (cf. Section 4).

\medskip
\noindent
{\bf Acknowledgement.} 
The second author was supported by JSPS KAKENHI Grant Number 16K05130.

\section{A special class of $G$-invariant metrics on $G/H$ and the Ricci tensor}

Let $G$ be a compact Lie group and $H$ a closed subgroup so that $G$ acts transitively on the homogeneous space $G/H$.  Then the homogeneous space $G/H$ is reductive, because we can take $\fr{m} = \fr{g}^{\perp}$ where $\Ad(H)\fr{m}\subset \fr{m}$ with respect to an $\Ad$-invariant scalar product on $\fr{g}$.  So the Lie algebra $\fr{g}$ can be written as $\fr{g} = \fr{h}\oplus\fr{m}$.  The tangent space of $G/H$ at the ${\it o} = eH\in G/H$ is canonically identified with $\fr{m}$.  For  a compact semisimple Lie group $G$ the negative of the Killing form $B$ of $\fr{g}$ is an $\Ad(G)$-invariant scalar product, therefore we can choose the above decomposition with respect to this form.  A Riemannian metric $g$ on $G/H$ is called $G$-invariant if the diffeomorphism $\tau_{\al} : G/H \to G/H,$ $\tau_{\al}(g H) = \al g H$ is an isometry.  We denote by $\mathcal{M}^{G}$ the set of all $G$-invariant metrics.  Any such a metric is to one-to-one correspondence with an $\Ad(H)$-invariant scalar product $\langle\cdot,\cdot\rangle$ on $\fr{m}$ and is considered as a point of fixed points $(\mathcal{M}^{G})^{\Phi_{H}}$ of the action $\Phi_{H}$ $=\{\Ad(h)|_{\fr{m}} : h\in H\}$ $\subset \Phi = \{\Ad(n)|_{\fr{m}} :  n \in N_{G}(H)\}$ $\subset$ $\Aut(\fr{m})$ on $\mathcal{M}^{G}$.  In the special case where $H = \{e\}$ then $N_{G}(H) = G$, thus the fixed points $(\mathcal{M}^{G})^{\Phi}$ are the $\Ad(G)$-invariant scalar products on $\fr{g}$.  These correspond to the bi-invariant metrics on the Lie group $G$. 

The isotropy representation $\chi : H \to \Aut(\fr{m})$ of the reductive homogeneous space $G/H$ coincides with the restriction of the adjoint representation of $H$ on $\fr{m}$.  We assume that $\chi$ decompose into a direct sum of irreducible subrepresentation $\chi \cong \chi_{1}\oplus\cdots\oplus\chi_{s}$, so the tangent space splits into a direct sum of $\Ad(H)$-invariant subspaces 
\begin{equation}\label{diaspasi}
\fr{m} = \fr{m}\oplus\cdots\oplus\fr{m}_{s}.
\end{equation}
In this case the $G$-invariant metrics are determined by the diagonal $\Ad(H)$-invariant scalar products of the form
\begin{equation}\label{diagonia}
\langle\cdot,\cdot\rangle = x_{1}(-B)|_{\fr{m}_{1}} + \cdots + x_{s}(-B)|_{\fr{m}_{s}}, 
\end{equation}
where $x_{i}\in\bb{R}^{+}$. 
If some of the subrepresentations $\chi_{i}$ are equivalent then decomposition (\ref{diaspasi}) is not unique.  Hence the $\Ad(H)$-invariant scalar product is not necessary diagonal.  For this case we can choose a closed subgroup $K$ of $G$ such that $H\subset K\subset N_{G}(H)$ and search for $\Ad(K)$-invariant scalar products on $\fr{m}$ which correspond to the fixed points $(\mathcal{M}^{G})^{\Phi_{K}}$ which is a subset of $\mathcal{M}^{G}$, sometimes also called $\Ad(K)$-invariant metrics.  The benefit of such metrics is that they are diagonal metrics on the homogeneous space.  The next proposition gives a possible way to choose such a subgroup $K$ of $G$.

\begin{prop}\label{subset}
Let $K$  be a subgroup of $G$ with $H\subset K \subset G$ and such that $K = L\times H$, for some subgroup $L$ of $G$.  Then $K$ is contained in $N_{G}(H)$.
\end{prop}
 
Now we describe the Ricci tensor for the diagonal metrics of the form (\ref{diagonia}).  Every $G$-invariant symmetric covariant 2-tensor on $G/H$ are of the same form as the Riemannian metrics (although they are not necessarly positive definite).  In particular, the Ricci tensor $r$ of a $G$-invariant Riemannian metric on $G/H$ is of the same form (\ref{diagonia}), that is
$$
r = r_1 x_1(-B)|_{\fr{m}_1} + \cdots +r_{s} x_{s}(-B)|_{\fr{m}_s}.
$$
Let $\lbrace e_{\alpha} \rbrace$ be a $(-B)$-orthonormal basis adapted to the decomposition of $\frak m$, i.e. $e_{\alpha} \in {\frak m}_i$ for some $i$, and $\alpha < \beta$ if $i<j$.  We put ${A^\gamma_{\alpha \beta}}= -B\left(\left[e_{\alpha},e_{\beta}\right],e_{\gamma}\right)$ so that $\left[e_{\alpha},e_{\beta}\right]
= \displaystyle{\sum_{\gamma} A^\gamma_{\alpha \beta} e_{\gamma}}$ and set $A_{ijk}:=\displaystyle{k \brack {ij}}=\sum (A^\gamma_{\alpha \beta})^2$, where the sum is taken over all indices $\alpha, \beta, \gamma$ with $e_\alpha \in {\frak m}_i,\ e_\beta \in {\frak m}_j,\ e_\gamma \in {\frak m}_k$ (cf. \cite{WZ}).  Then the positive numbers $A_{ijk}$ are independent of the 
$(-B)$-orthonormal bases chosen for ${\frak m}_i, {\frak m}_j, {\frak m}_k$, and 
$A_{ijk} = A_{jik} = A_{kij}$.  Let $ d_k= \dim{\frak m}_{k}$. Then we have the following:

\begin{lemma}\label{ric2}\textnormal{(\cite{PS})}
The components ${ r}_{1}, \dots, {r}_{s}$ of the Ricci tensor ${r}$ of the metric $ \langle\cdot, \cdot \rangle$ of the
form {\em (\ref{diagonia})} on $G/H$ are given by 
\begin{equation}
{r}_k = \frac{1}{2x_k}+\frac{1}{4d_k}\sum_{j,i}
\frac{x_k}{x_j x_i} A_{jik}
-\frac{1}{2d_k}\sum_{j,i}\frac{x_j}{x_k x_i} A_{kij}
 \quad (k= 1,\ \dots, s),    \label{eq51}
\end{equation}
where the sum is taken over $i, j =1,\dots, s$.
\end{lemma} 
Since by assumption the submodules $\fr{m}_{i}, \fr{m}_{j}$ in the decomposition (\ref{diaspasi}) are mutually non equivalent for any $i\neq j$, it is $r(\fr{m}_{i}, \fr{m}_{j})=0$ whenever $i\neq j$.  Thus by Lemma \ref{ric2} it follows that $G$-invariant Einstein metrics on $M=G/H$ are exactly the positive real solutions $g=(x_1, \ldots, x_s)\in\bb{R}^{s}_{+}$  of the  polynomial system $\{r_1=\lambda, \ r_2=\lambda,  \ldots,  r_{s}=\lambda\}$, where $\lambda\in \bb{R}_{+}$ is the Einstein constant.

\section{The Stiefel manifolds $V_k\bb{R}^{n}\cong\SO(n)/\SO(n-k)$}

We embed the group $\SO(n-k)$ in $\SO(n)$ as $\begin{pmatrix}
1_k & 0\\
0 & C
\end{pmatrix}$ where $C\in\SO(n-k)$.  The Killing form of $\fr{so}(n)$ is $B(X, Y)=(n-2)\tr XY$.
Then with respect to $-B$ the subspace $\fr{m}$ may be identified with the set 
of matrices of the form
$$
\begin{pmatrix}
D_k & A\\
-A^t & 0_{n-k}
\end{pmatrix},
$$
in $\fr{so}(n)$, where $D_k \in \fr{so}(k)$ 
 and $A$ is an $k\times (n-k)$ real matrix.  Let $E_{ab}$ denote the $n\times n$ matrix with $1$ at the $(ab)$-entry and $0$ elsewhere.  Then the set $\mathcal{B}=\{e_{ab}=E_{ab}-E_{ba}: 1\le a\le k,\ 1\le a<b\le n\}$
constitutes a $(-B)$-orthogonal basis of $\fr{m}$.  Note that $e_{ba}=-e_{ab}$, thus we have the following:

\begin{lemma}\label{brac}
If all four indices are distinct, then the Lie brakets in $\mathcal{B}$ are zero.
Otherwise,
$[e_{ab}, e_{bc}]=e_{ac}$, where $a,b,c$ are distinct.
\end{lemma}

The isotropy representation of $G/H = \SO(n)/\SO(n-k)$ given by 
\begin{equation}\label{Stiefel}
\chi = \underbrace{1\oplus \cdots\oplus 1}_{{k}\choose{2}}\oplus \underbrace{\lambda _{n-k}\oplus \cdots\oplus\lambda _{n-k}}_{k}
\end{equation}
where $\lambda_{n-k} : \SO(n-k)\to \Aut(\bb{R}^{n-k})$ be the standart representation of $\SO(n-k)$. 
This decomposition induces an $\Ad(H)$-invariant decomposition of $\fr{m}\cong T_{\it o}(G/H)$ given by

\begin{equation}\label{decomp}
\fr{m}=\fr{m}_1\oplus\cdots\oplus\fr{m}_s,
\end{equation}
where the first ${k}\choose{2}$ $\Ad (H)$-submodules are $1$-dimensional and the rest
$k$ $\Ad (H)$-submodules are $(n-k)$-dimensional.  Observe that the decomposition (\ref{Stiefel}) contains equivalent summands so a complete description of all $G$-invariant metrics associated to decomposition (\ref{decomp}) is rather hard.

\subsection{Stiefel manifolds as total spaces over  generalized flag manifolds.}
Recall that generalized flag manifolds whose isotropy representation has two irreducible summands were classified in \cite{ArCh} and \cite{Oh}.
Here we will use the generalized flag manifolds $M=G/K=
B_\ell(m) = \SO(2\ell + 1)/\U(\ell-m)\times\SO(2m+1)$, and $D_\ell(m) = \SO(2\ell)/\U(\ell-m)\times\SO(2m)$,  with two isotropy summands 
${\frak p}^{}_1, {\frak p}^{}_2$.
We will use these spaces to study the Stiefel manifolds $V_{2m+1}\bb{R}^{2\ell+1}$ and $V_{2m}\bb{R}^{2\ell}$.  
The painted Dynkin diagrams of these flag manifolds $M=G/K$ are defined by a pair $(\Pi, \Pi_{K})$, such that $\Pi\setminus \Pi_{K} = \{\al_{i_{0}}\}$ with ${\rm ht}(\al_{i_{0}}) = 2$ for some simple root $\al_{i_{0}}$.  Here $\Pi = \{\al_1,\dots, \al_{\ell}\}$ $(\dim(\fr{h}^{\bb{C}})=\ell)$ is a basis of simple roots of $G$, $\Pi_{K}\subset \Pi$ and ${\rm ht}(\al_{j})$ is the Dynkin mark of a simple root $\al_j$.  For more details see (\cite{Arv}).
 
We set $p = \ell - m$ and the above information is presented in the following table:

\vspace{0.4cm}
{\small 
\begin{center}
\begin{tabular}{c|c|c|c}
\hline 
\begin{picture}(20,30)(0,0)
\put(10, 12){\makebox(0,0){$G$}}\end{picture}
 &  \begin{picture}(30,30)(0,0)\put(10, 12){
\makebox(0,0){$( \Pi, \Pi_{K} ) $}}\end{picture}
& \begin{picture}(30,30)(0,0)\put(15, 12){
\makebox(0,0){$K$}}\end{picture}
 &\begin{picture}(30,30)(0,0)\put(15, 12){
\makebox(0,0){\shortstack{$\dim {\frak p}^{}_1$ \\ \\
 $\dim {\frak p}^{}_2$}}}\end{picture} \\
\hline 
\begin{picture}(15,40)(0,0)
\put(10, 20){\makebox(0,0){$B^{}_\ell$}}\end{picture}
 &
\begin{picture}(160,40)(-15,-22)
\put(0, 0){\circle{4}}
\put(0,8){\makebox(0,0){${\footnotesize \al_1}$}}
\put(2, 0){\line(1,0){14}}
\put(18, 0){\circle{4}}
\put(20, 0){\line(1,0){10}}
\put(18,8){\makebox(0,0){${\footnotesize \al_{2}}$}}
\put(40, 0){\makebox(0,0){$\ldots$}}
\put(50, 0){\line(1,0){10}}
\put(80, -16){\makebox(0,0){$( 2 \leq p \leq \ell-1 )$}}
\put(60, 0){\circle*{4.4}}
\put(60, 8){\makebox(0,0){${\footnotesize \al_{p}}$}}
\put(60, 0){\line(1,0){10}}
\put(80, 0){\makebox(0,0){$\ldots$}}
\put(90, 0){\line(1,0){10}}
\put(102, 0){\circle{4}}
\put(98, 8){\makebox(0,0){${\footnotesize \al_{\ell-1}}$}}
\put(103.5, 1.3){\line(1,0){15.5}}
\put(103.5, -1.3){\line(1,0){15.5}}
\put(115.5, -2.25){\scriptsize $>$}
\put(124, 0){\circle{4}}
\put(124, 8){\makebox(0,0){${\footnotesize \al_{\ell}}$}}
\end{picture}
 & 
\begin{picture}(110,40)(0, 6)
\put(50, 17){\makebox(5,15){$\U(p)\times\SO(2(\ell-p)+1)$}}\end{picture}
 &
 \begin{picture}(85,40)(0,6)
\put(37, 17){
\makebox(5,15){\shortstack{$2p(2(\ell - p)+1)$ \\ 
\\ $p(p - 1)$}}}\end{picture}
\\  \hline
\begin{picture}(15,40)(0,0)
\put(10, 20){\makebox(0,0){$D^{}_\ell$}}\end{picture}
 &
\begin{picture}(160,40)(-15,-21)
\put(0, 0){\circle{4}}
\put(0,8){\makebox(0,0){${\footnotesize \al_1}$}}
\put(2, 0){\line(1,0){14}}
\put(18, 0){\circle{4}}
\put(20, 0){\line(1,0){10}}
\put(18,8){\makebox(0,0){${\footnotesize \al_2}$}}
\put(40, 0){\makebox(0,0){$\ldots$}}
\put(50, 0){\line(1,0){10}}
\put(70, -15){\makebox(0,0){$( 2 \leq p \leq \ell -2 )$}}
\put(60, 0){\circle*{4.4}}
\put(60, 8){\makebox(0,0){${\footnotesize \al_p}$}}
\put(60, 0){\line(1,0){10}}
\put(80, 0){\makebox(0,0){$\ldots$}}
\put(90, 0){\line(1,0){10}}
\put(102, 0){\circle{4}}
\put(103.7, 1){\line(2,1){10}}
\put(103.7, -1){\line(2,-1){10}}
\put(115.5, 6){\circle{4}}
\put(115.5, -6){\circle{4}}
\put(132, 11){\makebox(0,0){${\footnotesize \al_{\ell-1}}$}}
\put(118, -14){${\footnotesize \al_{\ell}}$}
\end{picture}
 &
\begin{picture}(110,40)(5,24)\put(50, 25){
\makebox(8,35){$\U(p)\times\SO(2(\ell-p))$}}\end{picture}
 &
 \begin{picture}(80,40)(5,24)\put(30, 25){
\makebox(10,35){\shortstack{$4 p (\ell - p)$ \\ \\ 
$p(p - 1)$}}}\end{picture}
\\  \hline
\end{tabular}
\end{center} }
\vspace{0.4cm}

We consider the fibrations
$$
\U(p) \to \SO(2\ell + 1)/\SO(2(p+1)) \to \SO(2\ell+1)/\U(p)\times\SO(2(\ell-p)+1) = B_\ell(p)
$$
and
$$
\U(p) \to \SO(2\ell)/\SO(2p) \to \SO(2\ell)/\U(p)\times\SO(2(\ell-p)) = D_\ell(p)
$$
For both cases the tangent space $\fr{m}$ of the Stiefel manifold (total space) can be written as follows
$$
\fr{m} = \fr{u}(p)\oplus\fr{p}_1\oplus\fr{p}_2
$$
where $\fr{p}_1\oplus\fr{p}_2$ is the tangent space of generalized flag manifold. 
 We know that 
  the Lie algebra $\fr{u}(p)$ of $\U(p)$  is written as $\fr{h}_0\oplus\fr{h}_1$ where $\fr{h}_0$ is the center of $\fr{u}(p)$ and $\fr{h}_1 = \fr{su}(p)$.  
We set $n=2\ell$ or $2\ell+1$.  
  Then  the tangent space $\fr{m}$ decomposes into mutually non equivalent  irreducible $\Ad(\U(p)\times\SO(n- 2 p))$-submodules 
$$
\fr{m} = \fr{h}_0\oplus\fr{h}_1\oplus\fr{p}_1\oplus\fr{p}_2 \equiv \fr{m}_0\oplus\fr{m}_1\oplus\fr{m}_2\oplus\fr{m}_3.
$$
 We consider $G$-invariant metrics on Stiefel manifold defined by  the $\Ad(\U(p)\times\SO(n- 2 p))$-invariant scalar products  on $\fr{m}$   given by  
\begin{eqnarray}\label{metric}
\langle\cdot,\cdot\rangle = u_0(-B)|_{\fr{m}_0} + u_1(-B)|_{\fr{m}_1} + u_2(-B)|_{\fr{m}_2} + u_3(-B)|_{\fr{m}_3}
\end{eqnarray}
where $u_i \in \bb{R}_{+}$ ($ i= 0,1,2,3$).   

We set $d_1 = \dim(\fr{m}_1)$, $d_2 = \dim(\fr{m}_2)$ and $d_3 = \dim(\fr{m}_3)$.  It easy to see that the following relations hold:
\begin{eqnarray}\label{LieB}
[\fr{m}_2, \fr{m}_2]\subset \fr{m}_0\oplus\fr{m}_1\oplus\fr{m}_3, \ \ [\fr{m}_3, \fr{m}_3]\subset \fr{m}_0\oplus\fr{m}_1, \ \ [\fr{m}_2, \fr{m}_3]\subset\fr{m}_2.
\end{eqnarray}

\subsection{The Ricci tensor for metrics corresponding to scalar products (\ref{metric}).}

From (\ref{LieB}) and  [\cite{ArMoSa}, Proposition 6, p.\, 269] we see that the only non zero numbers triples (up to permutation of indices) for the metric corresponding to (\ref{metric}) are
$$
A_{220}, \quad A_{330}, \quad A_{111}, \quad A_{122}, \quad A_{133}, \quad A_{322}.
$$

\begin{lemma}\label{brac2}{\em (\cite{ArMoSa})} The non zero triples $A_{ijk}$ are given as follows:
\begin{eqnarray*}
\begin{array}{lll}
\displaystyle{A_{220} = \frac{d_2}{(d_2 + 4d_3)}}, & \displaystyle{A_{330} = \frac{4d_3}{(d_2 + 4d_3)}}, &
\displaystyle{A_{111} = \frac{2d_3(2d_1 + 2 - d_3)}{(d_2 +4d_3)}},  \vspace{3pt}\\ 
\displaystyle{A_{122} = \frac{d_1d_2}{(d_2 + 4d_3)}}, &
\displaystyle{A_{133} = \frac{2d_3(d_3 - 2)}{(d_2 + 4 d_3)}}, & \displaystyle{A_{322} = \frac{d_2d_3}{(d_2 + 4d_3)}}.
\end{array} 
\end{eqnarray*}
\end{lemma}

By using the above lemma, we obtain the components of the Ricci tensor for the metric (\ref{metric}) give as follows:

\begin{prop}\label{prop1}
The components  of  the Ricci tensor ${r}$ for the invariant metric $\langle\cdot ,\cdot\rangle $ on Stiefel manifold $G/H$ defined by  $(\ref{metric})$ are given as follows$:$  
\begin{equation}\label{ricci2}
\begin{array}{lll} 
r_{0} &=& \displaystyle{\frac{u_0}{4u_{2}^{2}}\,\frac{d_{2}}{(d_2 + 4d_3)} + \frac{u_0}{4u_{3}^{2}}\,\frac{4d_3}{(d_2 + 4d_3)}},\\ \\
r_{1} &=& \displaystyle{\frac{1}{4d_{1}u_{1}}\,\frac{2d_3(2d_1 + 2 -d_3)}{(d_2 + 4d_3)} + \frac{u_1}{4 u_{2}^{2}}\,\frac{d_2}{(d_2 + 4d_3)} + \frac{u_1}{2d_1 u_{3}^{2}}\,\frac{d_3(d_3 - 2)}{(d_2 + 4d_3)}},\\ \\
r_{2} &=& \displaystyle{\frac{1}{2u_2} -\frac{u_3}{2 u_{2}^{2}}\,\frac{ d_3}{(d_2 + 4d_3)} -\frac{1}{2u_{2}^{2}}\left(u_{0}\,\frac{1}{(d_2 + 4d_3)} + u_{1}\,\frac{d_{1}}{(d_2 + 4d_3)}\right)}, \\ \\
r_{3} &=& \displaystyle{\frac{1}{u_3}\,\left(\frac{1}{2} - \frac{1}{2}\,\frac{d_2}{(d_2 + 4d_3)}\right) + \frac{u_3}{4u_{2}^{2}}\,\frac{d_2}{(d_2 + 4d_3)} 
- \frac{1}{ u_{3}^{2}}\left(u_{0}\,\frac{2}{(d_2 + 4d_3)} + u_{1}\,\frac{d_3 - 2}{(d_2 + 4d_3)}\right)}.\\ \\
\end{array}  
\end{equation}
\end{prop}
The metric of the form (\ref{metric}) is Einstein if and only if the system of equations 
\begin{equation}\label{system1}
r_0 = r_1, \ r_1 = r_2, \ r_2 = r_3,
\end{equation}
has positive solutions.

\section{Einstein metrics on Stiefel manifolds}

In this section we study the Einstein metrics on Stiefel manifolds when these are total space over the flag manifolds by solving  the system  of equations (\ref{system1}).


\begin{theorem}
For $ \displaystyle 2 \leq p \leq \frac{2}{5} n-1$,   the Stiefel manifold $V_{2 p}\bb{R}^{n} = \SO(n)/\SO(n-2 p)$ admits at least four invariant Einstein metrics which are determined by the $\Ad(\U(p)\times\SO(n - 2 p) )$-invariant scalar products. Two  of the metrics are Jensen's metrics and the other two are different from Jensen's metrics.
\end{theorem}

\begin{proof}
We have $d_1 = p^2-1, d_2 = 2 p (n-2 p)$ and $d_3 = p(p-1)$.  After normalizing with $u_3 = 1$, the system of equations (\ref{system1})  becomes
\begin{eqnarray}\label{system4}
f_1 & = & {u_0} {u_1} (n-2 p)-{u_1}^2 (n-2 p)+2 (p-1)
   {u_0} {u_1} {u_2}^2-(p-2) {u_1}^2 {u_2}^2-p{u_2}^2 =0, \nonumber\\
f_2 &= &{u_1}^2 \left(n p-p^2-1\right)-2 (n-2) p
   {u_1} {u_2}+p^2 {u_2}^2+(p-2) p {u_1}^2
   {u_2}^2 \nonumber\\
   & &+(p-1) p {u_1}+{u_0} {u_1}=0, \\
f_3& = & 2 (n-2) p
   {u_2}+p (-n+p+1)+2 (p-2) (p+1) {u_1} {u_2}^2-(p-1)
   (p+1) {u_1} \nonumber\\
   & & -4 (p-1) p {u_2}^2+4 {u_0} {u_2}^2-{u_0} =0. \nonumber 
\end{eqnarray}

We consider a polynomial ring $R = \bb{Q}[n, p][z, u_0, u_1, u_2]$ and an ideal $J$ generated by $\{f_1, f_2, f_3,$ $ z\,u_0\,u_1\,u_2 - 1\}$ to find non zero solutions of equation (\ref{system4}).  We take the lexicographic order $>$ with $z>u_0>u_1>u_2$ for a monomial ordering on $R$.  Then, by an aid of computer, we see that a Gr\"obner basis for the ideal $J$ contains the polynomial $(2 (p-1) {u_2}^2-2 (n-2) {u_2}+n-1)G_{n, p}(u_2)$, where

{\small\begin{eqnarray*}
&& G_{n, p}(u_2) = 8 (p-2) (p-1) \left(5 p^3-8 p^2+p-2\right) {u_2}^8-8 (n-2) (p-2) (3 p-1) \left(p^2-p+2\right) {u_2}^7 \\ & &
+4  \left(17 n p^4-58 n p^3+61 n p^2-20 n p+4 n-30 p^5+101 p^4-88
   p^3-11 p^2+28 p-4\right){u_2}^6 \\ & &
-8 (n-2) \left(4 p^3-9 p^2+9 p-6\right) (n-2 p){u_2}^5 +2 (p-1) (n-2 p)(21 n p^2-26 n p-32 p^3\\
&& +34 p^2+32 p-24){u_2}^4 -2 (n-2) \left(7 p^2-8 p+4\right) (n-2 p)^2{u_2}^3 + (n-2 p)^2(11 n p^2-12 n p\\
&& -14 p^3+9 p^2+20 p-14){u_2}^2-2 (n-2) p (n-2 p)^3 {u_2}  +(n-2 p)^3 \left(n p-p^2-p+1\right). 
\end{eqnarray*}
}
If $2(p-1)u_2^2-2(n-2)u_2+n-1=0$, we obtain Jensen's Einstein metrics.

\noindent
Note that  the polynomial $ G_{n, p}(u_2)$ can be expressed as 
{\small\begin{eqnarray*}
&& G_{n, p}(u_2) =8\left(5 (p-2)^3+22 (p-2)^2+29 (p-2)+8\right)(p-2)(p-1) {u_2}^8\\ & & 
-8 (n-2) (p-2) (3 p-1) \left(p^2-p+2\right) {u_2}^7+((68(p-2)^4+312 (p-2)^3+484 (p-2)^2\\
&& +288(p-2)+64) (n-2 p)  +16 (p-2)^5 +100(p-2)^4+296 (p-2)^3+452 (p-2)^2\\
&& +320 (p-2)+96){u_2}^6  -8 (n-2) \left(4 (p-2)^3+15 (p-2)^2+21 (p-2)+8\right) (n-2 p){u_2}^5 \\ & & 
+ 2 (n - 2 p) ( p -1)\left( p (21 p-26) (n-2 p)+10 (p-2)^3+42 (p-2)^2+80 (p-2)+48\right){u_2}^4  \\ & & 
-2 (n-2) \left(7 p^2-8 p+4\right) (n-2 p)^2{u_2}^3  +(p (11 p-12)(n-2 p)+ 8 (p-2)^3+33 (p-2)^2\\
&& +56(p-2) +30)(n-2 p)^2{u_2}^2 - \left(2 p (n-2 p)+4 (p-1) p \right)(n-2 p)^3 u_2 +\\
&& p (n-2 p)^4+\left(p^2-p+1\right)(n-2 p)^3. 
\end{eqnarray*}
}
From the above expression we see that if the equation $ G_{n, p}(u_2) =0$ has  real solutions,  then these are positive for $2 \leq p \leq n/2$. 

Now we take the lexicographic order $>$ with $z>u_1>u_2>u_0$ for a monomial ordering on $R$.  Then, by the aid of computer, we see that a Gr\"obner basis for the ideal $J$ contains the polynomial $(u_0-1)H_{n, p}(u_0)$ where 
$$H_{n, p}(u_0)= \sum_{k=0}^{8}a_{k}^{}(n, p){u_0}^k$$
{\small \begin{eqnarray*}
&&\boldsymbol{a_{0}(n, p)} =64 (p-2) p^3 (n-{(5 (p+1))}/{2})^9+96 (p-2) p^3
   (11 p+7) (n-{(5 (p+1))}/{2})^8\\
   &&+576 (p-2) p^3
   (p+1) (13 p+5) (n-{(5 (p+1))}/{2})^7+\big(30004 (p-2)^7+427488
   (p-2)^6\\ &&
   +2515972 (p-2)^5+7835288 (p-2)^4+13626592
   (p-2)^3+12555872 (p-2)^2+4791744 (p-2)\\ &&
   +512\big)
   (n-{(5 (p+1))}/{2})^6+4
   \big(18789 (p
   -2)^6+248313 (p-2)^5+1295975 (p-2)^4\\
   &&+3340687
   (p-2)^3+4256148 (p-2)^2+2146136 (p-2)+928\big) p^2
   (n-{(5 (p+1))}/{2})^5
\\
     &&
   +\big(122679
   (p-2)^7+1995834 (p-2)^6+13365614 (p-2)^5+47170188
   (p-2)^4\\
   &&
   +92568607 (p-2)^3+95837594 (p-2)^2+40950780
   (p-2)+44512\big) p^2 (n-{(5 (p+1))}/{2})^4\\
   &&+2 \big(65563
   (p-2)^9+1402351 (p-2)^8+12986098 (p-2)^7+68013362
   (p-2)^6\\
   &&
   +220420041 (p-2)^5+452819519 (p-2)^4+576171516
   (p-2)^3+415509440 (p-2)^2\\&&
   +130343854 (p-2)+282848\big) p
   (n-{(5 (p+1))}/{2})^3+\frac{1}{4} \big(354867
   (p-2)^{10}+8723046 (p-2)^9\\
   &&+94326525 (p-2)^8+589028208
   (p-2)^7+2341360757 (p-2)^6+6145435310 (p-2)^5\\
   &&+10655244691
   (p-2)^4+11775095484 (p-2)^3+7534446728 (p-2)^2+2135654080
   (p-2)\\ &&
   +8047168\big) p (n-{(5 (p+1))}/{2})^2+\frac{1}{4} \big(3936 (p-2)^8+75201
   (p-2)^7+603308 (p-2)^6\\
&&
   +2638758 (p-2)^5+6806032
   (p-2)^4+10370489 (p-2)^3+8665044 (p-2)^2+3082288
   (p-2)\\ &&
   +18432\big) p (p+1) (5 p+1) (7 p+11) (n-{(5 (p+1))}/{2})+\frac{1}{16} \big(92455
   (p-2)^{11}+2535522 (p-2)^{10}\\
&&   
   +30740389 (p-2)^9+217376892
   (p-2)^8+994341005 (p-2)^7+3078217258 (p-2)^6\\ &&
   +6539013051
   (p-2)^5+9423500944 (p-2)^4+8830656620 (p-2)^3+4872616168
   (p-2)^2\\ &&+1215054848 (p-2)+10612736\big) (p+1)^2,\\
&& \boldsymbol{a_{1}(n, p)} =  -4 (p-1)\big(128 (p-2)^6+1216 (p-2)^5+4672
   (p-2)^4+9088 (p-2)^3+8960 (p-2)^2\\ 
   &&+3584
   (p-2)\big) (n-{5 p}/{2}-1)^9+\big(2112 (p-2)^7+23904
   (p-2)^6+113568 (p-2)^5\\ 
   &&
   +290112 (p-2)^4+420480 (p-2)^3+327936 (p-2)^2+107520 (p-2)\big)
   (n-{5 p}/{2}-1)^8 \\ 
   &&
   +\big(14916
   (p-2)^8+197064 (p-2)^7+1121380
   (p-2)^6+3564272 (p-2)^5+6836320
   (p-2)^4 \\ 
   &&
+7914240 (p-2)^3+5121856
   (p-2)^2+1431296 (p-2)+2048\big)
   (n-{5 p}/{2}-1)^7 \\ 
   & &
+\big(59258
   (p-2)^9+900046 (p-2)^8+5999444
   (p-2)^7+22929778 (p-2)^6+54974414
   (p-2)^5 \\ 
   &&
+84683468 (p-2)^4+81877672
   (p-2)^3+45473968 (p-2)^2+11164800
   (p-2)+52224\big)\times \\ 
   &&
 (n-{5 p}/{2}-1)^6+\big(146457
   (p-2)^{10}+2526600 (p-2)^9+19402335
   (p-2)^8+87072194 (p-2)^7
\\
   &&
+251721306
   (p-2)^6+486278748 (p-2)^5+627970968
   (p-2)^4+523159664 (p-2)^3\\ 
   &&
+255759424
   (p-2)^2+56605184 (p-2)+561152\big) (n-{5 p}/{2}-1)^5+\big(\frac{469341}{2}
   (p-2)^{11}\\ 
   &&
+\frac{9103821}{2}
   (p-2)^{10}+39742074
   (p-2)^9+\frac{411452227}{2}
   (p-2)^8+\frac{1398838889}{2}
   (p-2)^7\\ 
   &&
+1632279289 (p-2)^6+2649253192
   (p-2)^5+2954811746 (p-2)^4+2170833280
   (p-2)^3\\
   &&
+953204568 (p-2)^2+194506752
   (p-2)+3295744\big) (n-{5 p}/{2}-1)^4+\big(\frac{978763}{4} (p-2)^{12}
 \\ 
   &&
+\frac{10583245}{2}
   (p-2)^{11}+\frac{207883447}{4}
   (p-2)^{10}+306072569 (p-2)^9+1201323551
   (p-2)^8\\ 
   &&
+3300567873 (p-2)^7+6479707714
   (p-2)^6+9095524741 (p-2)^5+8956549778
   (p-2)^4\\ 
   &&
+5908577049 (p-2)^3+2369183094
   (p-2)^2+454508872 (p-2)+11444480\big)
   (n-{5 p}/{2}-1)^3
\\   &&
+\Big(\frac{1280703}{8}
   (p-2)^{13}+\frac{30677505}{8}
   (p-2)^{12}+\frac{168071501}{4}
   (p-2)^{11}+\frac{2228550051}{8}
   (p-2)^{10}\\ 
   &&
+\frac{9959361633}{8}
   (p-2)^9+\frac{15808695135}{4}
   (p-2)^8+\frac{18285571615}{2}
   (p-2)^7+\frac{31075995511}{2}
   (p-2)^6     \end{eqnarray*} 
 \begin{eqnarray*}   
   &&
+19271465037
   (p-2)^5+\frac{34085400131}{2}
   (p-2)^4+10243332195 (p-2)^3+3802685000
   (p-2)^2\\ 
   &&
+696906320 (p-2)+23552512\Big)
   (n-{5 p}/{2}-1)^2+\Big(\frac{947715}{16}
   (p-2)^{14}+1565412
   (p-2)^{13}
\\
   &&
+\frac{304515821}{16}
   (p-2)^{12}+\frac{1127961691}{8}
   (p-2)^{11}+\frac{5682777937}{8}
   (p-2)^{10}+2571059035 (p-2)^9\\ 
   &&
+6880922151
   (p-2)^8+\frac{55178382763}{4}
   (p-2)^7+\frac{41465027375}{2}(p-2)^6+\frac{92412916219}{4}
   (p-2)^5\\ 
   &&
+\frac{37205172749}{2}
   (p-2)^4+10309661636 (p-2)^3+3584080158
   (p-2)^2+635127812 (p-2)\\ 
   &&
+26665248\Big)
   (n-{5 p}/{2}-1)+294175 (p-2)^{15}+8592575 (p-2)^{14}+115880340
   (p-2)^{13}\\ 
   &&
+956898953 (p-2)^{12}+5408093279
   (p-2)^{11}+22140703898 (p-2)^{10}+67749542208
   (p-2)^9\\ 
   &&
+157506396024 (p-2)^8+279787494680
   (p-2)^7+378339776596 (p-2)^6\\ 
   &&
+384097487656
   (p-2)^5+284893169184 (p-2)^4+147081785600
   (p-2)^3\\ 
   &&
+48336089440 (p-2)^2+8351483328
   (p-2)+410291712),
\\
&&\boldsymbol{a_{2}(n, p)} =
\Big(192 (p-2)^7+2048 (p-2)^6+9408
   (p-2)^5
+23936 (p-2)^4+35584 (p-2)^3\\
&& +29184
   (p-2)^2+10240 (p-2)\Big) (n-{5 p}/{2}-1)^9+\Big(3312 (p-2)^8+41216
   (p-2)^7\\
&& +225136 (p-2)^6+702304 (p-2)^5+1353408
   (p-2)^4+1610368 (p-2)^3+1094144
   (p-2)^2   \\
&& +329728 (p-2)+4096 \Big)(n-{5 p}/{2}-1)^8+\big(24208
   (p-2)^9+347216 (p-2)^8+2216816
   (p-2)^7\\  
&& +8253232 (p-2)^6+19637056
   (p-2)^5+30606400 (p-2)^4+30536960
   (p-2)^3
\\
  && +17896960 (p-2)^2+4845568
   (p-2)+135168\big) (n-{5 p}/{2}-1)^7+\big(98528
   (p-2)^{10}      \\ 
&& +1613184 (p-2)^9+11883208
   (p-2)^8+51819304 (p-2)^7+147723352
   (p-2)^6  \\
   && +285946328 (p-2)^5+376382880
   (p-2)^4+325765344 (p-2)^3+169876480
   (p-2)^2\\
&& +42795776 (p-2)+1897472\big)
   (n-{5 p}/{2}-1)^6+\big(247124
   (p-2)^{11}+4580924 (p-2)^{10}
\\
&& +38523772
   (p-2)^9+193968644 (p-2)^8+648837864
   (p-2)^7+1509338408 (p-2)^6\\
&& +2477084616
   (p-2)^5+2839067064 (p-2)^4+2186568576
   (p-2)^3+1038899424 (p-2)^2\\
&& +249609216
   (p-2)+14834688\big) (n-{5 p}/{2}-1)^5+\big(398490
   (p-2)^{12}+8301156 (p-2)^{11}\\
&& +78968564
   (p-2)^{10}+453731874 (p-2)^9+1752889850
   (p-2)^8+4789573762 (p-2)^7\\
&& +9462416340
   (p-2)^6+13547348796 (p-2)^5+13826101616
   (p-2)^4+9658457712 (p-2)^3\\
&& +4254867968
   (p-2)^2+989303552 (p-2)+70789120\big)
   (n-{5 p}/{2}-1)^4+\big(414783
   (p-2)^{13}\\
&& +9649107 (p-2)^{12}+103029649
   (p-2)^{11}+668978113 (p-2)^{10}+2947340816
   (p-2)^9\\
&& +9299133844 (p-2)^8+21584813268
   (p-2)^7+37217960968 (p-2)^6+47451436288
   (p-2)^5\\
&& +43848137836 (p-2)^4+28181444408
   (p-2)^3+11651671384 (p-2)^2+2638899872
   (p-2)\\
&& +211548672\big) (n-{5 p}/{2}-1)^3+\big(268330
   (p-2)^{14}+6941678
   (p-2)^{13}+\frac{165440483}{2}
   (p-2)^{12}\\
&& +\frac{1204853103}{2}
   (p-2)^{11}+\frac{5994590717}{2}
   (p-2)^{10}+\frac{21557653433}{2}
   (p-2)^9+28879727380 (p-2)^8\\
&& +58490990190
   (p-2)^7+89775704240 (p-2)^6+103513336538
   (p-2)^5+87706716076 (p-2)^4\\
&& +52409556188
   (p-2)^3+20501025840 (p-2)^2+4529087712
   (p-2)+387504960\big) \times \\
&& (n-{5 p}/{2}-1)^2+\big(96760
   (p-2)^{15}+\frac{11132569}{4}
   (p-2)^{14}+\frac{73857043}{2}
   (p-2)^{13}
\\ 
&& +\frac{1201391141}{4}
   (p-2)^{12}+\frac{3353198347}{2}
   (p-2)^{11}+\frac{13619138187}{2}
   (p-2)^{10}+\frac{41589778805}{2}
   (p-2)^9
        \end{eqnarray*} 
 \begin{eqnarray*}  
&& +\frac{97215269053}{2}
   (p-2)^8+87601832480 (p-2)^7+121498189843
   (p-2)^6+128215240066 (p-2)^5\\
&& +100586689666
   (p-2)^4+56320881064 (p-2)^3+20952659436
   (p-2)^2+4510207392 (p-2)\\
&& +398720448\big)
   (n-{5 p}/{2}-1) +    \frac{226935}{16} (p-2)^{16}+\frac{1834381}{4}
   (p-2)^{15}+\frac{109089629}{16}
   (p-2)^{14}\\
&& +62091126
   (p-2)^{13}+\frac{3109271349}{8}
   (p-2)^{12}+\frac{14210852695}{8}
   (p-2)^{11}+\frac{24559554309}{4}
   (p-2)^{10}\\
&& +\frac{65532732125}{4}
   (p-2)^9+34111002493
   (p-2)^8+\frac{111147048281}{2}
   (p-2)^7+70516477093 (p-2)^6\\
&& +68787658421
   (p-2)^5
+50367541042 (p-2)^4+26588687170
   (p-2)^3+9441545008 (p-2)^2\\
&& +1976735544
   (p-2)+ 176734800,\\
&& \boldsymbol{a_{3}(n, p)} = -4 \Big(\big(24 (p-2)^8+280
   (p-2)^7+1592 (p-2)^6+5704 (p-2)^5+13456
   (p-2)^4\\
&&+20352 (p-2)^3+18624 (p-2)^2+9344
   (p-2)+2048\big)(n-{5 p}/{2}-1)^8
   +\big(332 (p-2)^9\\
&&+4564
   (p-2)^8+29988 (p-2)^7+123740 (p-2)^6+346352
   (p-2)^5+661712 (p-2)^4
\\
&&+839648 (p-2)^3+670400
   (p-2)^2+304640 (p-2)+61440\big)
   (n-{5 p}/{2}-1)^7+\big(1934
   (p-2)^{10}     \\ 
&&+30856 (p-2)^9+232702
   (p-2)^8+1097914 (p-2)^7+3563344
   (p-2)^6+8176750 (p-2)^5\\
&&+13179572
   (p-2)^4+14496112 (p-2)^3+10306512
   (p-2)^2+4272128 (p-2)+795136\big)\times
   \\
&&
   (n-{5 p}/{2}-1)^6+\big(6243
   (p-2)^{11}+114103 (p-2)^{10}+979277
   (p-2)^9+5241915 (p-2)^8\\
&&+19431720
   (p-2)^7+51975418 (p-2)^6+101131872
   (p-2)^5+141340908 (p-2)^4\\
&&+137595368
   (p-2)^3+88287704 (p-2)^2+33588592
   (p-2)+5791744\big) (n-{5 p}/{2}-1)^5\\
&&+\Big(\frac{24685}{2}
   (p-2)^{12}+255270
   (p-2)^{11}+\frac{4937957}{2}
   (p-2)^{10}+\frac{29722031}{2}
   (p-2)^9\\
&&+62156304 (p-2)^8+\frac{379748941}{2}
   (p-2)^7+431586299 (p-2)^6+729072080
   (p-2)^5\\
&&+900445208 (p-2)^4+787219568
   (p-2)^3+460409744 (p-2)^2+161703520
   (p-2)\\
&&+25947008\Big)(n-{5 p}/{2}-1)^4+\Big(\frac{61809}{4}
   (p-2)^{13}+\frac{1430259}{4}
   (p-2)^{12}+\frac{15434635}{4}
   (p-2)^{11}\\
&&+\frac{103486953}{4}
   (p-2)^{10}+120747132 (p-2)^9+414689966
   (p-2)^8+1075702486 (p-2)^7\\
&&+2123142483
   (p-2)^6+3167874581 (p-2)^5+3506287672
   (p-2)^4+2783402730 (p-2)^3\\
&&+1495868084
   (p-2)^2+487731980 (p-2)+73169728\Big)
   (n-{5 p}/{2}-1)^3+\Big(\frac{96373}{8}
   (p-2)^{14}\\
&&+308800
   (p-2)^{13}+\frac{29475725}{8}
   (p-2)^{12}+\frac{218233347}{8}
   (p-2)^{11}+\frac{281388471}{2}
   (p-2)^{10}\\
&&+\frac{4291030137}{8}
   (p-2)^9+\frac{6239274315}{4}
   (p-2)^8+\frac{7017001009}{2}
   (p-2)^7+\frac{12220342497}{2}
   (p-2)^6\\
&&+8158484460 (p-2)^5+8180580411
   (p-2)^4+5948629046 (p-2)^3+2957904604
   (p-2)^2\\
&&+900067032 (p-2)+126827872\Big)
   (n-{5 p}/{2}-1)^2+\Big(\frac{85977}{16}
   (p-2)^{15}+\frac{2420673}{16}
   (p-2)^{14}\\
&&+\frac{31693567}{16}
   (p-2)^{13}+\frac{257285909}{16}
   (p-2)^{12}+\frac{363802367}{4}
   (p-2)^{11}+\frac{3049303845}{8}
   (p-2)^{10} \\
&&+1225906604
   (p-2)^9+\frac{6166824267}{2}
   (p-2)^8+\frac{24429920569}{4}
   (p-2)^7+\frac{18999705075}{2}
   (p-2)^6\\
&&+\frac{22930968755}{2}
   (p-2)^5+10502822300 (p-2)^4+7044898860
   (p-2)^3+3259688412 (p-2)^2\\
&&+930061998
   (p-2) +123620640\Big) (n-{5 p}/{2}-1)+\frac{8405}{8} (p-2)^{16}+\frac{1032749}{32}
   (p-2)^{15} \\
&&+\frac{1844589}{4}
   (p-2)^{14}+\frac{16339575}{4}
   (p-2)^{13}+\frac{403385501}{16}
   (p-2)^{12}+\frac{3693426499}{32}
   (p-2)^{11}
        \end{eqnarray*} 
 \begin{eqnarray*}  
&&+\frac{6509124153}{16}
   (p-2)^{10}+\frac{9030801709}{8}
   (p-2)^9+\frac{19956847161}{8}
   (p-2)^8+\frac{8805249399}{2}
   (p-2)^7\\
&&+\frac{24692355419}{4}
   (p-2)^6+6785397574
   (p-2)^5+\frac{11428567897}{2}
   (p-2)^4+3554405017 (p-2)^3\\
&&+1537253850
   (p-2)^2+412857018 (p-2)+51949224\Big),
 \\
 &&
  \boldsymbol{a_{4}(n, p)} =\Big(16 (p-2)^8+160
   (p-2)^7+784 (p-2)^6+2624 (p-2)^5+7040
   (p-2)^4\\
&&+14720 (p-2)^3+20224 (p-2)^2+14848
   (p-2)+4096\Big) (n-{5 p}/{2}-1)^8+16 \big(15 (p-2)^9\\
&&+179
   (p-2)^8+1039 (p-2)^7+4035 (p-2)^6+12010
   (p-2)^5+27902 (p-2)^4+47020 (p-2)^3  \\
&&+51304
   (p-2)^2+31056 (p-2)+7552\big)
   (n-{5 p}/{2}-1)^7+4 \big(383
   (p-2)^{10}+5346 (p-2)^9\\
&&+35980 (p-2)^8+159526
   (p-2)^7+530549 (p-2)^6+1374600
   (p-2)^5+2697272 (p-2)^4\\
&&+3761480
   (p-2)^3+3420080 (p-2)^2+1771904
   (p-2)+384512\big) (n-{5 p}/{2}-1)^6\\
&&+4 \big(1366
   (p-2)^{11}+21955 (p-2)^{10}+168839
   (p-2)^9+844397 (p-2)^8+3123141
   (p-2)^7\\
&&+8982894 (p-2)^6+20027898
   (p-2)^5+33426974 (p-2)^4+39622784
   (p-2)^3\\
&&+30970568 (p-2)^2+14069760
   (p-2)+2752960\big) (n-{5 p}/{2}-1)^5+
   \big(16 (p-2)^8+160
   (p-2)^7\\
&&+784 (p-2)^6+2624 (p-2)^5+7040
   (p-2)^4+14720 (p-2)^3+20224 (p-2)^2+14848
   (p-2)\\
&&+4096\big) (n-{5 p}/{2}-1)^8+16 \big(15 (p-2)^9+179
   (p-2)^8+1039 (p-2)^7+4035 (p-2)^6\\
&&+12010
   (p-2)^5+27902 (p-2)^4+47020 (p-2)^3+51304
   (p-2)^2+31056 (p-2)+7552\big)\\
&&
  \times (n-{5 p}/{2}-1)^7+4 \big(383
   (p-2)^{10}+5346 (p-2)^9+35980 (p-2)^8+159526
   (p-2)^7\\
&&+530549 (p-2)^6+1374600
   (p-2)^5+2697272 (p-2)^4+3761480
   (p-2)^3+3420080 (p-2)^2\\
&&+1771904
   (p-2)+384512\big) (n-{5 p}/{2}-1)^6+4 \big(1366
   (p-2)^{11}+21955 (p-2)^{10}\\
&&+168839 (p-2)^9 +844397 (p-2)^8+3123141
   (p-2)^7+8982894 (p-2)^6+20027898
   (p-2)^5\\
&&+33426974 (p-2)^4+39622784
   (p-2)^3+30970568 (p-2)^2+14069760
   (p-2)+2752960\big) \\
&&\times(n-{5 p}/{2}-1)^5+\big(11999
   (p-2)^{12}+219174 (p-2)^{11}+1903518
   (p-2)^{10}+10634538 (p-2)^9\\
&&+43429619
   (p-2)^8+137504616 (p-2)^7+342175628
   (p-2)^6+658691852 (p-2)^5\\
&&+949147712
   (p-2)^4+977302432 (p-2)^3+671120864
   (p-2)^2+271962624 (p-2)\\
&&+48406016\big)
   (n-{5 p}/{2}-1)^4+\big(16759
   (p-2)^{13}+343985 (p-2)^{12}+3340385
   (p-2)^{11}\\
&&+20676953 (p-2)^{10}+92585532
   (p-2)^9+319829130 (p-2)^8+874746888
   (p-2)^7\\
&&+1891803264 (p-2)^6+3174682356
   (p-2)^5+4008952104 (p-2)^4+3647078460
   (p-2)^3\\
&&+2235676712 (p-2)^2+818518832
   (p-2)+133524416\big) (n-{5 p}/{2}-1)^3+\frac{1}{4} \big(58673
   (p-2)^{14}\\
&&+1339910 (p-2)^{13}+14425988
   (p-2)^{12}+98289426 (p-2)^{11}+479950435
   (p-2)^{10}\\
&&+1796016376 (p-2)^9+5330844880
   (p-2)^8+12683173064 (p-2)^7+24018443768
   (p-2)^6\\
&&+35510016400 (p-2)^5+39797188664
   (p-2)^4+32405254704 (p-2)^3+17943792656
   (p-2)^2\\
&&+5992714240 (p-2)+901550848\big)
   (n-{5 p}/{2}-1)^2+\frac{1}{4}
   \big(29684 (p-2)^{15}+747737
   (p-2)^{14}\\
  &&+8864377 (p-2)^{13}+66171015
   (p-2)^{12}+351338709 (p-2)^{11}+1419039674
   (p-2)^{10}
   \\
   &&+4533523474 (p-2)^9+11679070922
   (p-2)^8+24343088940 (p-2)^7+40677153072
   (p-2)^6
\\
&&+53456639604 (p-2)^5+53677185432
   (p-2)^4+39483510160 (p-2)^3+19914519952
   (p-2)^2\\
&&+6108577376 (p-2)+851422080\big)
   (n-{5 p}/{2}-1)+1681 (p-2)^{16}+\frac{92701}{2}
   (p-2)^{15}\\
&&+\frac{9624449}{16}
   (p-2)^{14}+\frac{39225997}{8}
   (p-2)^{13}+\frac{452561293}{16}
   (p-2)^{12}+123328536
   (p-2)^{11}
 \end{eqnarray*}  
 \begin{eqnarray*}  
&&+\frac{1692729139}{4}
   (p-2)^{10}+\frac{4682236169}{4}
   (p-2)^9+\frac{5278688557}{2}
   (p-2)^8+4850874951
   (p-2)^7\\
&&+\frac{14392593795}{2}
   (p-2)^6+8457199877 (p-2)^5+7652938727
   (p-2)^4+5112762958 (p-2)^3\\
&&+2360032842
   (p-2)^2+667419000 (p-2)+86396976,
\\
&&\boldsymbol{a_{5}(n, p)} =-4
   (2 n-p-1) \Big(8 \big((p-2)^7+10 (p-2)^6+49 (p-2)^5+155
   (p-2)^4+321 (p-2)^3\\
&&+400 (p-2)^2+260
   (p-2)+64\big) (n-{5 p}/{2}-1)^6+4 \big(21 (p-2)^8+260
   (p-2)^7+1510 (p-2)^6\\
&&+5483 (p-2)^5+13453
   (p-2)^4+21969 (p-2)^3+22276 (p-2)^2+12284
   (p-2)+2688\big)\\
&& \times (n-{5 p}/{2}-1)^5+2 \big(182 (p-2)^9+2689
   (p-2)^8+18259 (p-2)^7+76079 (p-2)^6\\
&&+215531
   (p-2)^5+425751 (p-2)^4+572467 (p-2)^3+490530
   (p-2)^2+235420 (p-2)\\
&&+46336\big)
   (n-{5 p}/{2}-1)^4+\big(818
   (p-2)^{10}+14072 (p-2)^9+110274
   (p-2)^8+525339 (p-2)^7\\
&&+1701320
   (p-2)^6+3921028 (p-2)^5+6460682
   (p-2)^4+7392461 (p-2)^3+5501018
   (p-2)^2\\
&&+2342716 (p-2)+419040\big)
   (n-{5 p}/{2}-1)^3+\frac{1}{2}
   \big(1957 (p-2)^{11}+38644 (p-2)^{10}\\
&&+346581
   (p-2)^9+1882392 (p-2)^8+6943765
   (p-2)^7+18384176 (p-2)^6+35686224
   (p-2)^5\\
&&+50498391 (p-2)^4+50433990
   (p-2)^3+33278324 (p-2)^2+12771720
   (p-2)+2095200\big)\\
&&\times  (n-{5 p}/{2}-1)^2+\frac{1}{4} \big(2281
   (p-2)^{12}+51484 (p-2)^{11}+527352
   (p-2)^{10}+3266850 (p-2)^9\\
&&+13736619
   (p-2)^8+41592491 (p-2)^7+93444338
   (p-2)^6+157129167 (p-2)^5\\
&&+195489718
   (p-2)^4+173862124 (p-2)^3+103425856
   (p-2)^2+36231808 (p-2)+5496960\big)
   \\
&& \times (n-{5 p}/{2}-1)+\frac{1}{8}
   \big(980 (p-2)^{13}+25433 (p-2)^{12}+298535
   (p-2)^{11}+2115318 (p-2)^{10}\\
&&+10166455
   (p-2)^9+35229193 (p-2)^8+91083041
   (p-2)^7+178648663 (p-2)^6\\
&&+266188790
   (p-2)^5+296805300 (p-2)^4+238960192
   (p-2)^3+129940440 (p-2)^2\\
&&+42015120
   (p-2)+5942016\big)\bigg),
\\
&&\boldsymbol{a_{6}(n, p)} =2 (2 n-p-1)^2 \Big(\big(8 (p-2)^5+56
   (p-2)^4+192 (p-2)^3+416 (p-2)^2+512
   (p-2)\\
&&+256\big) (n-{5 p}/{2}-1)^5+4 \big(19 (p-2)^6+177
   (p-2)^5+754 (p-2)^4+1926 (p-2)^3\\
&&+3090
   (p-2)^2+2848 (p-2)+1128\big)
   (n-{5 p}/{2}-1)^4+2 \big(135 p^7-325 p^6+699 p^5-299 p^4\\
&&+288
   p^3-210 p^2+52 p-16\big) (n-{5 p}/{2}-1)^3\big(421
   (p-2)^8+5907 (p-2)^7+37027 (p-2)^6\\
&&+136825
   (p-2)^5+329646 (p-2)^4+534610 (p-2)^3+570856
   (p-2)^2+364584 (p-2)\\
&&+105320\big)
   (n-{5 p}/{2}-1)^2+\frac{1}{2}
   \big(524 (p-2)^9+8904 (p-2)^8+67489
   (p-2)^7+301841 (p-2)^6\\
&&+885632 (p-2)^5+1784478
   (p-2)^4+2488660 (p-2)^3+2325156
   (p-2)^2+1317904 (p-2)\\
&&+342720\big)
   (n-{5 p}/{2}-1)+\frac{1}{4}
   \big(172 (p-2)^{10}+3752 (p-2)^9+35611
   (p-2)^8+197237 (p-2)^7\\
&&+715048 (p-2)^6+1792730
   (p-2)^5+3180636 (p-2)^4+3978628
   (p-2)^3+3377736 (p-2)^2\\
&&+1759056
   (p-2)+424800\big)\Big),
\\
&&
  \boldsymbol{a_{7}(n, p)} =   -4 (2 n-p-1)^3 \big(2 p (n-2 p)+p^2-p+2\big)\Big(\big(2 (p-2)^3+10 (p-2)^2+14
   (p-2)+4\big) \\   &&\times(n-{5 p}/{2}-1)^2+\big(6 (p-2)^4+41
   (p-2)^3+99 (p-2)^2+92 (p-2)+18\big)
   (n-{5 p}/{2}-1)\\
&&+\frac{5}{2}
   (p-2)^5+24 (p-2)^4+90 (p-2)^3+154 (p-2)^2+104
   (p-2)+6\Big),  
\\
&&
  a_{8}(n, p) =(2 n-p-1)^4 \left(2 p (n-2 p)+p^2-p+2\right)^2. 
    \end{eqnarray*}  }
   Thus we see that, if the equation $ H_{n, p}(u_0) =0$ has  real solutions,  then these are  positive for $ \displaystyle 2 \leq p \leq \frac{2}{5} n -1$.

Now we take the lexicographic order $>$ with $z>u_0>u_2>u_1$ for a monomial ordering on $R$.  Then, by the aid of computer, we see that a
 Gr\"obner basis for the ideal $J$ contains the polynomial $(u_1-1)F_{n, p}(u_1)$, where 
{\small \begin{eqnarray*}
&&\boldsymbol{F_{n, p}(u_1)}= (2 n-p-1)^4 (p-2)^2 (p+1)^2 {u_ 1}^8\\
& &
-2 (2 n-p-1)^3 (p-2) (p+1) \left(-10
   p^3+n p^2+25 p^2-20 p+4 n\right) {u_ 1}^7\\
   &&
   -2 (2 n-p-1) \big(-182 p^7-83 n
   p^6+1407 p^6+505 n^2 p^5-1587 n p^5-2570 p^5-284 n^3 p^4\\
   &&-603 n^2 p^4+5381 n
   p^4+504 p^4+36 n^4 p^3+710 n^3 p^3-1440 n^2 p^3-4595 n p^3+1377 p^3-64 n^4
   p^2\\
   &&-616 n^3 p^2+2768 n^2 p^2-250 n p^2+202 p^2+32 n^4 p+48 n^3 p-688 n^2
   p+344 n p-480 p+32 n^4
 \\  &&-112 n^3
   +112 n^2+64 n+32\big) {u_ 1}^6
 +(2 n-p-1)^2 \big(140 p^6-10 n
   p^5-726 p^5-99 n^2 p^4+414 n p^4\\
   &&+1057 p^4+44 n^3 p^3+68 n^2 p^3   -930 n p^3-178
   p^3-112 n^3 p^2+312 n^2 p^2+420 n p^2-596 p^2\\ &&+80 n^3 p-208 n^2 p-280 n p+816
   p-64 n^3+272 n^2-304 n-64\big) {u_ 1}^5+\big(46
   p^8+1262 n p^7
 \\ &&
   -2706 p^7-2685 n^2 p^6+2570 n p^6+7173 p^6+828 n^3 p^5+10434
   n^2 p^5-31616 n p^5+5706 p^5\\ &&
   +1996 n^4 p^4-20068 n^3 p^4+34511 n^2 p^4+15458 n
   p^4-15743 p^4-1728 n^5 p^3+9248 n^4 p^3\\ &&
   +7624 n^3 p^3-75052 n^2 p^3+63050 n
   p^3-14418 p^3+384 n^6 p^2+864 n^5 p^2-22832 n^4 p^2\\ &&+65888 n^3 p^2-54384 n^2
   p^2+16444 n p^2-1358 p^2-896 n^6 p+6336 n^5 p-12288 n^4 p+2384 n^3 p\\ &&
   +7616 n^2
   p-6256 n p+1784 p+256 n^6-2048 n^5+5632 n^4-6624 n^3+4336 n^2-1632
   n+248\big) {u_ 1}^4\\ &&
   -2 \big(-330 p^8+137 n p^7+2129 p^7+852 n^2 p^6-4328
   n p^6-2144 p^6-1120 n^3 p^5+2522 n^2 p^5\\ &&
   +6676 n p^5-2146 p^5+704 n^4 p^4-1288
   n^3 p^4-2508 n^2 p^4-4002 n p^4+5788 p^4-304 n^5 p^3\\ &&
   +1096 n^4 p^3-2600 n^3
   p^3+10706 n^2 p^3-14817 n p^3+3701 p^3+64 n^6 p^2-224 n^5 p^2+1104 n^4
   p^2\\ &&
   -6616 n^3 p^2+13188 n^2 p^2-6654 n p^2+1114 p^2-64 n^6 p-128 n^5 p+3008
   n^4 p-8600 n^3 p\\ &&
   +8552 n^2 p-4428 n p+1000 p+128 n^6-1024 n^5+2944 n^4-3952
   n^3+3024 n^2-1288 n+256\big) {u_ 1}^3\\ &&
   +\big(444 p^8-782 n p^7-1226
   p^7+363 n^2 p^6+2566 n p^6-121 p^6+232 n^3 p^5-2758 n^2 p^5+708 n p^5\\ &&
   +1854
   p^5-292 n^4 p^4+1524 n^3 p^4+447 n^2 p^4-5066 n p^4+439 p^4+80 n^5 p^3-160
   n^4 p^3-1948 n^3 p^3\\ &&
   +5916 n^2 p^3-1638 n p^3-646 p^3-96 n^5 p^2+896 n^4
   p^2-1760 n^3 p^2-1512 n^2 p^2+4028 n p^2\\ &&
   -1020 p^2-320 n^4 p+1792 n^3 p-2768
   n^2 p+600 n p+80 p-224 n^3+896 n^2-848 n+320\big) {u_ 1}^2
  \\ &&
   +2 \big(3 p^2-2 n p+p-2\big) \big(18 p^6-35 n p^5-5 p^5+31 n^2 p^4-17 n p^4-10
   p^4-16 n^3 p^3+37 n^2 p^3 \\
    &&
   -25 n p^3+62 p^3+4 n^4 p^2-18 n^3 p^2+26 n^2 p^2-53
   n p^2+9 p^2+4 n^3 p-16 n^2 p+44 n p-56 p+16\big) {u_ 1}\\ &&
   +\big(p^2-n p+p-1\big)^2 \left(3 p^2-2 n p+p-2\right)^2. 
     \end{eqnarray*}   }     
If $u_1=1$ we obtain Jensen's Einstein metrics.
      Moreover, the Gr\"obner basis contains polynomials of the form 
 $b(n, p) u_0 - X(u_1)$ and $c(n, p) u_2 - Y(u_1)$, where 
 $X(u_1)$ and $Y(u_1)$ are some polynomials of degree 7 with coefficients in $\bb{Q}[n, p]$, and $b(n, p)$ and $c(n, p)$ are in $\bb{Q}[n, p]$. In fact, 
 {\small
  \begin{eqnarray*}
&&
\boldsymbol{b(n, p)}=  -32 \left(2  p( n -2 p) + p^2-p+2\right)^2\\ && \times\Big(\big(1600 (p-3)^6+22400 (p-3)^5+126400 (p-3)^4+364800 (p-3)^3+561600 (p-3)^2\\
&&+432000 (p-3)+129600\big) (n-2 p)^{14}+\big(24992 (p-3)^7+397120 (p-3)^6+2639648 (p-3)^5\\
&&+9476032 (p-3)^4+19750624 (p-3)^3+23787712 (p-3)^2+15280224 (p-3)+4051008\big)\\
&&\times (n-2 p)^{13}+\big(172336 (p-3)^8+3060400 (p-3)^7+23345632 (p-3)^6+99714576 (p-3)^5\\
&&+260277216 (p-3)^4+424326352 (p-3)^3+421471968
   (p-3)^2+233456240 (p-3)
        \end{eqnarray*} 
 \begin{eqnarray*}  
 &&+55552560\big) (n-2 p)^{12}+\big(689538 (p-3)^9+13512892 (p-3)^8+116029868 (p-3)^7
\\
&&+572333580
   (p-3)^6+1785443216 (p-3)^5+3650392180 (p-3)^4+4891103124 (p-3)^3\\
&&+4146940228 (p-3)^2+2026691246 (p-3)+438563968\big) (n-2
   p)^{11}+\big(1773104 (p-3)^{10}\\
&&+37938152 (p-3)^9+361157714 (p-3)^8+2013129174 (p-3)^7+7272879190 (p-3)^6\\
&&+17791197494
   (p-3)^5+29857709290 (p-3)^4+33995657362 (p-3)^3+25214056658 (p-3)^2\\
&&+11064354170 (p-3)+2201283692\big) (n-2 p)^{10}+\big(3078144
   (p-3)^{11}+71263024 (p-3)^{10}
\\
&&+743009164 (p-3)^9+4603448388 (p-3)^8+18827485508 (p-3)^7+53377113932 (p-3)^6\\
&&+107111336372
   (p-3)^5+152376423724 (p-3)^4+151042751404 (p-3)^3+99828711028 (p-3)^2\\
   &&+39865401488 (p-3)+7349429344\big) (n-2 p)^9+\big(3717188
   (p-3)^{12}+92477036 (p-3)^{11}\\
   &&+1046494337 (p-3)^{10}+7120830760 (p-3)^9+32445195504 (p-3)^8+104306889892 (p-3)^7\\
&&+242779969176
   (p-3)^6+412828294298 (p-3)^5+510279831304 (p-3)^4+448872403836 (p-3)^3\\
&&+268133641731 (p-3)^2+98253767954 (p-3)+16796669704\big) (n-2
   p)^8+\big(3217600 (p-3)^{13}\\
&&+85779040 (p-3)^{12}+1049391440 (p-3)^{11}+7798553312 (p-3)^{10}+39276103438 (p-3)^9\\
&&+141588772948
   (p-3)^8+376103441332 (p-3)^7+746204224328 (p-3)^6+1107844903214 (p-3)^5\\
&&+1218694176544 (p-3)^4+968572561004 (p-3)^3+528301042196
   (p-3)^2+177608181876 (p-3)\\
&&+27774685648\big) (n-2 p)^7+\big(2085040 (p-3)^{14}+59814784 (p-3)^{13}+794400772 (p-3)^{12}\\
&&+6473317332
   (p-3)^{11}+36158180946 (p-3)^{10}+146483415104 (p-3)^9+444034228432 (p-3)^8\\
&&+1023797180622 (p-3)^7+1805798361308 (p-3)^6+2426997252450
   (p-3)^5+2448365609082 (p-3)^4\\
&&+1797632672456 (p-3)^3+906613659852 (p-3)^2+279936682548 (p-3)+39584942472\big) (n-2 p)^6\\
&&+\big(1057152 (p-3)^{15}+32936192 (p-3)^{14}+479344544 (p-3)^{13}+4322774160 (p-3)^{12}\\
&&+27014619096 (p-3)^{11}+123933631612 (p-3)^{10}+431229952100
   (p-3)^9+1159006157180 (p-3)^8\\
&&+2426151768832 (p-3)^7+3955312197368 (p-3)^6+4979245956600 (p-3)^5+4749414262148 (p-3)^4\\
&&+3317658181124
   (p-3)^3+1598429426380 (p-3)^2+473308267496 (p-3)+64634577456\big) (n-2 p)^5\\
&&+\big(415520 (p-3)^{16}+14113984 (p-3)^{15}+225708864
   (p-3)^{14}+2255397056 (p-3)^{13}\\
&&+15759395378 (p-3)^{12}+81638285772 (p-3)^{11}+324297094399 (p-3)^{10}+1007589196068
   (p-3)^9\\
&&+2474476711162 (p-3)^8+4819249988316 (p-3)^7+7418923490891 (p-3)^6+8933651083982 (p-3)^5\\
&&+8251575670178 (p-3)^4+5654842061140
   (p-3)^3+2714716062858 (p-3)^2\\
&&+817342571780 (p-3)+116662053632\big) (n-2 p)^4+\big(116576 (p-3)^{17}+4298624 (p-3)^{16}
\\
&&
+75130192
   (p-3)^{15}+826412720 (p-3)^{14}+6406269046 (p-3)^{13}+37134701924 (p-3)^{12}
\\   
&&+166662234436 (p-3)^{11}+591546075680
   (p-3)^{10}+1681266530110 (p-3)^9+3849160646768 (p-3)^8
\\
&&+7101622506196 (p-3)^7+10505119842624 (p-3)^6+12317098218390
   (p-3)^5\\
&&+11218811444660 (p-3)^4+7676477338608 (p-3)^3+3724468474664 (p-3)^2\\
&&+1146181916270 (p-3)+168794152112\big) (n-2 p)^3+\big(21120
   (p-3)^{18}+842304 (p-3)^{17} \\
&&+16026624 (p-3)^{16}+193240064 (p-3)^{15}+1654078424 (p-3)^{14}+10670996900 (p-3)^{13}\\
&&+53761273072
   (p-3)^{12}+216249222138 (p-3)^{11}+703992041434 (p-3)^{10}+1868772806378 (p-3)^9\\
&&+4055046321312 (p-3)^8+7176809956664
   (p-3)^7+10285658097412 (p-3)^6\\
&&+11778667276216 (p-3)^5+10544080113596 (p-3)^4+7122124018534 (p-3)^3\\
&&+3419349157538 (p-3)^2+1041531393170
   (p-3)+151507442700\big) (n-2 p)^2+\big(2176 (p-3)^{19}\\
&&+93568 (p-3)^{18}+1929920 (p-3)^{17}+25362816 (p-3)^{16}+237944008
   (p-3)^{15}+1692175912 (p-3)^{14}\\
&&+9455173680 (p-3)^{13}+42456694960 (p-3)^{12}+155403427848 (p-3)^{11}+467578535064
   (p-3)^{10}
        \end{eqnarray*} 
 \begin{eqnarray*}  
&&+1160829638848 (p-3)^9+2377230479528 (p-3)^8+3998145895696 (p-3)^7+5472923178272 (p-3)^6\\
&&+6007447014744 (p-3)^5+5166301657560
   (p-3)^4+3355990390400 (p-3)^3+1549560080000 (p-3)^2  \\
&&+453536215000 (p-3)+63289950000\big) (n-2 p)+64 (p-3)^{20}+2880 (p-3)^{19}+62448 (p-3)^{18}\\
&&+866592 (p-3)^{17}+8623068 (p-3)^{16}+65342356 (p-3)^{15}+390920801 (p-3)^{14}+1889383324
   (p-3)^{13}\\
&&+7487294598 (p-3)^{12}+24552266432 (p-3)^{11}+66947671665 (p-3)^{10}+151978346834 (p-3)^9\\
&& + 286581722768 (p-3)^8+446227977020
   (p-3)^7+567858927275 (p-3)^6+581249565000 (p-3)^5\\
&&+467104073750 (p-3)^4+283883000000 (p-3)^3+122677171875 (p-3)^2+33587656250
   (p-3)\\
&&+4378125000\Big),
 \end{eqnarray*} 
 }
for $p \geq 3$ and for $p=2$ it is 
{\small \begin{eqnarray*} &&
\boldsymbol{b(n, 2)}= -8192 (n-3)^2 (n-2)^4\big(1024 (n-5)^9+17408 (n-5)^8+123904 (n-5)^7+475968 (n-5)^6
\\&&
+1057600 (n-5)^5
+1350064 (n-5)^4+943152 (n-5)^3+411260 (n-5)^2
+263700 (n-5)+135409\big).
 \end{eqnarray*} }
 Also, 
 {\small
  \begin{eqnarray*}
&&
\boldsymbol{c(n, p)}=  64 (n-2) \left(5 (p-2)^3+22 (p-2)^2+29 (p-2)+8\right) (p-1) (n-2 p) \left(2 p (n-2 p)+p^2-p+2\right) 
\\
&&
\Big( \big(1600 (p-3)^6+22400 (p-3)^5+126400
   (p-3)^4+364800 (p-3)^3+561600 (p-3)^2\\
&&
+432000 (p-3)+129600\big) (n-2 p)^{14}+\big(24992 (p-3)^7+397120 (p-3)^6+2639648
   (p-3)^5\\
&&+9476032 (p-3)^4+19750624 (p-3)^3+23787712 (p-3)^2+15280224 (p-3)+4051008\big) (n-2 p)^{13}\\
&&+\big(172336
   (p-3)^8+3060400 (p-3)^7+23345632 (p-3)^6+99714576 (p-3)^5+260277216 (p-3)^4\\
&&
+424326352 (p-3)^3+421471968
   (p-3)^2+233456240 (p-3)+55552560\big) (n-2 p)^{12}+\big(689538 (p-3)^9\\
&&
+13512892 (p-3)^8+116029868 (p-3)^7+572333580
   (p-3)^6+1785443216 (p-3)^5
 \\&&
+3650392180 (p-3)^4+4891103124 (p-3)^3+4146940228 (p-3)^2+2026691246 (p-3)+438563968\big)\\
&&
\times 
   (n-2 p)^{11}+\big(1773104 (p-3)^{10}+37938152 (p-3)^9+361157714 (p-3)^8+2013129174 (p-3)^7\\
&&
+7272879190
   (p-3)^6+17791197494 (p-3)^5+29857709290 (p-3)^4+33995657362 (p-3)^3\\
&&
+25214056658 (p-3)^2+11064354170
   (p-3)+2201283692\big) (n-2 p)^{10}+\big(3078144 (p-3)^{11}\\
&&
+71263024 (p-3)^{10}+743009164 (p-3)^9+4603448388
   (p-3)^8+18827485508 (p-3)^7  \\&&
+53377113932 (p-3)^6+107111336372 (p-3)^5+152376423724 (p-3)^4+151042751404
   (p-3)^3  \\
&&
+99828711028 (p-3)^2+39865401488 (p-3)+7349429344\big) (n-2 p)^9+\big(3717188 (p-3)^{12}
\\
&&
+92477036
   (p-3)^{11}+1046494337 (p-3)^{10}+7120830760 (p-3)^9+32445195504 (p-3)^8
 \\&&
+104306889892 (p-3)^7+242779969176
   (p-3)^6+412828294298 (p-3)^5+510279831304 (p-3)^4
\\ 
&&
+448872403836 (p-3)^3+268133641731 (p-3)^2+98253767954
   (p-3)+16796669704\big) (n-2 p)^8\\
&&
+\big(3217600 (p-3)^{13}+85779040 (p-3)^{12}+1049391440 (p-3)^{11}+7798553312
   (p-3)^{10} \\
&&
+39276103438 (p-3)^9+141588772948 (p-3)^8+376103441332 (p-3)^7+746204224328 (p-3)^6\\
&&
+1107844903214
   (p-3)^5+1218694176544 (p-3)^4+968572561004 (p-3)^3+528301042196 (p-3)^2\\
&&
+177608181876 (p-3)+27774685648\big) (n-2 p)^7+\big(2085040 (p-3)^{14}+59814784 (p-3)^{13}  
\\
&&
+794400772 (p-3)^{12}+6473317332 (p-3)^{11}+36158180946
   (p-3)^{10}+146483415104 (p-3)^9\\
&&
+444034228432 (p-3)^8+1023797180622 (p-3)^7+1805798361308 (p-3)^6+2426997252450
   (p-3)^5\\
&&
+2448365609082 (p-3)^4+1797632672456 (p-3)^3+906613659852 (p-3)^2+279936682548 (p-3)\\
&&
+39584942472\big) (n-2
   p)^6+\big(1057152 (p-3)^{15}+32936192 (p-3)^{14}+479344544 (p-3)^{13} \\
&&
+4322774160 (p-3)^{12}+27014619096
   (p-3)^{11}+123933631612 (p-3)^{10}+431229952100 (p-3)^9
      \end{eqnarray*}  
  \begin{eqnarray*} 
&&
+1159006157180 (p-3)^8+2426151768832 (p-3)^7+3955312197368
   (p-3)^6+4979245956600 (p-3)^5\\
&&
+4749414262148 (p-3)^4+3317658181124 (p-3)^3+1598429426380 (p-3)^2+473308267496
   (p-3)\\
&&
+64634577456\big) (n-2 p)^5+\big(415520 (p-3)^{16}+14113984 (p-3)^{15}+225708864 (p-3)^{14}\\
&&
+2255397056
   (p-3)^{13}+15759395378 (p-3)^{12}+81638285772 (p-3)^{11}+324297094399 (p-3)^{10}\\
&&
+1007589196068 (p-3)^9+2474476711162
   (p-3)^8+4819249988316 (p-3)^7
+7418923490891 (p-3)^6\\
&&
+8933651083982 (p-3)^5+8251575670178 (p-3)^4+5654842061140
   (p-3)^3+2714716062858 (p-3)^2\\
&&
+817342571780 (p-3)+116662053632\big) (n-2 p)^4+\big(116576 (p-3)^{17}+4298624
   (p-3)^{16}\\
&&
+75130192 (p-3)^{15}+826412720 (p-3)^{14}+6406269046 (p-3)^{13}+37134701924 (p-3)^{12}\\
&&
+166662234436
   (p-3)^{11}+591546075680 (p-3)^{10}+1681266530110 (p-3)^9+3849160646768 (p-3)^8\\
&&
+7101622506196 (p-3)^7+10505119842624
   (p-3)^6+12317098218390 (p-3)^5\\
&&
+11218811444660 (p-3)^4+7676477338608 (p-3)^3+3724468474664 (p-3)^2+1146181916270
   (p-3)\\
&&
+168794152112\big) (n-2 p)^3+\big(21120 (p-3)^{18}+842304 (p-3)^{17}+16026624 (p-3)^{16}\\
&&
+193240064
   (p-3)^{15}+1654078424 (p-3)^{14}+10670996900 (p-3)^{13}+53761273072 (p-3)^{12}\\
&&
+216249222138 (p-3)^{11}+703992041434
   (p-3)^{10}+1868772806378 (p-3)^9+4055046321312 (p-3)^8\\
&&
+7176809956664 (p-3)^7+10285658097412 (p-3)^6+11778667276216
   (p-3)^5\\
&&
+10544080113596 (p-3)^4+7122124018534 (p-3)^3+3419349157538 (p-3)^2+1041531393170 (p-3)\\
&&
+151507442700\big) (n-2
   p)^2+\big(2176 (p-3)^{19}+93568 (p-3)^{18}+1929920 (p-3)^{17}+25362816 (p-3)^{16}\\
&&
+237944008 (p-3)^{15}+1692175912
   (p-3)^{14}+9455173680 (p-3)^{13}+42456694960 (p-3)^{12}+\\
&&
155403427848 (p-3)^{11}+467578535064 (p-3)^{10}+1160829638848
   (p-3)^9+2377230479528 (p-3)^8\\
&&
+3998145895696 (p-3)^7+5472923178272 (p-3)^6+6007447014744 (p-3)^5+5166301657560
   (p-3)^4\\
&&
+3355990390400 (p-3)^3+1549560080000 (p-3)^2+453536215000 (p-3)+63289950000\big) (n-2 p)   \\
&&
+64 (p-3)^{20}+2880 (p-3)^{19}+62448 (p-3)^{18}+866592 (p-3)^{17}+8623068 (p-3)^{16}  \\
&&
+65342356 (p-3)^{15}+390920801
   (p-3)^{14}+1889383324 (p-3)^{13}+7487294598 (p-3)^{12}\\
&&
+24552266432 (p-3)^{11}+66947671665 (p-3)^{10}+151978346834
   (p-3)^9+286581722768 (p-3)^8\\
&&
+446227977020 (p-3)^7+567858927275 (p-3)^6+581249565000 (p-3)^5+467104073750
   (p-3)^4\\
&&
+283883000000 (p-3)^3+122677171875 (p-3)^2+33587656250 (p-3)+4378125000
    \end{eqnarray*} }
    for $p \geq 3$ and for $p=2$ it is 
{\small \begin{eqnarray*} &&
\boldsymbol{c(n, 2)}= 32768 (n-4) (n-3) (n-2)^5
\big(1024 (n-5)^9+17408 (n-5)^8+123904 (n-5)^7\\&&+475968 (n-5)^6+1057600 (n-5)^5
+1350064 (n-5)^4+943152 (n-5)^3+411260
   (n-5)^2\\&&+263700 (n-5)+135409\big).    
    \end{eqnarray*} }
    
Thus we see that   $b(n, p)$ and $c(n, p)$  are  non zero  for $n- 2p > 0$ and $p \geq 2$. 
  In particular, if $u_1$ are reals, so are the solutions $ u_2 $ and  $u_0$.
     
     Now see that 
 {\small \begin{eqnarray*}
&&\boldsymbol{F_{n, p}(1)}=256 n^6 p^2-768 n^6 p-1024 n^5 p^3+1024 n^5 p^2+6784 n^5 p-384 n^5+1024 n^4 p^4
\\&&
   +4608 n^4 p^3   -20480 n^4 p^2-19456 n^4 p+1792 n^4+1024 n^3 p^5-11776 n^3 p^4+12032 n^3 p^3
   \\&&+66944 n^3 p^2
   +21376 n^3 p-2304 n^3-2048 n^2 p^6+5120 n^2
   p^5+21248 n^2 p^4-69120 n^2 p^3 \\
   &&-79552 n^2 p^2-7296 n^2 p   +576 n^2+7168 n p^6-28160 n p^5+13568 n p^4+79232 n p^3+31488 n p^2 \\
&&
   +1152 n p+1024 p^8 -6144 p^7+8704 p^6+9216 p^5-22976 p^4-20352 p^3-4032 p^2
     \end{eqnarray*}  
  \begin{eqnarray*}    &&
 = \big(256 (p-3)^2+768 (p-3)\big) (n-{(5 (p+1)})/{2})^6+\big(2816
   (p-3)^3+18688 (p-3)^2+31360 (p-3)\\&& +1536\big) (n-{(5 (p+1)})/{2})^5+
   \big(12224 (p-3)^4+127296 (p-3)^3+442688 (p-3)^2+527872
   (p-3)
\\   &&
   +43264\big) (n-{(5 (p+1)})/{2})^4+\big(27264
   (p-3)^5+389504 (p-3)^4+2081600 (p-3)^3\\ &&
   +4978048 (p-3)^2+4670336
   (p-3)+478720\big) (n-{(5 (p+1)})/{2})^3+\big(34032
   (p-3)^6\\ &&
   +616656 (p-3)^5+4454688 (p-3)^4+16104896 (p-3)^3+29472000
   (p-3)^2+22949632 (p-3)\\ && 
   +2631936\big) (n-{(5 (p+1)})/{2})^2+\big(22960 (p-3)^7+502288 (p-3)^6+4559624
   (p-3)^5\\&&
   +22039488 (p-3)^4+60136288 (p-3)^3+89046528 (p-3)^2+59364352
   (p-3)\\ &&
   +7186432\big) (n-{(5 (p+1)})/{2})+6724 (p-3)^8+171052
   (p-3)^7+1856868 (p-3)^6\\ && 
   +11166320 (p-3)^5+40284304 (p-3)^4+87708288
   (p-3)^3+108449792 (p-3)^2\\ && +62996480 (p-3)+7770112  > 0
   \end{eqnarray*} }
   for $ \displaystyle 3 \leq p \leq \frac{2}{5}n -1$. 
 
 \medskip  
     Note that for $p =2$ we have 
 {\small     \begin{eqnarray*}&&
      \boldsymbol{F_{n, 2}(2 n)} = 65536 (n-7)^9+3645440 (n-7)^8+89782272 (n-7)^7+1285090304 (n-7)^6 \\
      &&
      +11781647616
   (n-7)^5+71750921344 (n-7)^4+290281447040 (n-7)^3\\ &&+752322349200
   (n-7)^2+1133434256000 (n-7)+756320640000 > 0.
   \end{eqnarray*} }
We also see that 
{ \small \begin{eqnarray*}
&&\boldsymbol{F_{n, p}}\Big(\frac{1}{4}\Big)= -\big(\left(32768 p^2+98304 p+196608\right) (n-{(5 (p+1)})/{2})^6+256
   \big(985 p^3+6524 p^2\\ && 
   +9932 p+5424\big) (n-{(5 (p+1)})/{2})^5+16 \big(27175 p^4+504198 p^3+1094251 p^2+855884
   p\\ && +216092\big) (n-{(5 (p+1)})/{2})^4+32 \big(13929 p^5+377600
   p^4+1776269 p^3+2029438 p^2
\\ &&
   +894812 p+119752\big) (n-{(5 (p+1)})/{2})^3+16 \big(20277 p^6+517836 p^5+3876897 p^4+9354262
   p^3  \\ &&
   +6363072 p^2+1645640 p+140016\big) (n-{(5 (p+1)})/{2})^2+64
   \big(1709 p^7+41993 p^6+409393 p^5  \\ &&
   +1646739 p^4+2428106 p^3+1038900
   p^2+179240 p+10880\big) (n-{(5 (p+1)})/{2})
   +64 \big(101 p^8  \\  
   &&+4130 p^7+52794 p^6+312052 p^5+802181 p^4+741498 p^3+231340 p^2+30160
   p+1424\big)
\big) < 0 
  \end{eqnarray*} }
   for $ \displaystyle 2 \leq p \leq \frac{2}{5}n -1$ and 
   $$F_{n, p} (0 ) = \big(p^2-n p+p-1\big)^2 \left(3 p^2-2 n p+p-2\right)^2 > 0. $$ 
   
  Thus we obtain at least two positive solutions for the equation $F_{n, p} (u_1 )  =0$ for $ \displaystyle 2 \leq p \leq \frac{2}{5}n -1$. \end{proof}
  
   \begin{remark}
  For $n = 31$ we have $2 \leq p \leq  2 n/5 -1 = 62/5 -1 = 11.4$, but for $p =12, 13$ we see that, except Jensen's Einstein metrics, for 
 $V_{26}\mathbb{R}^{31} = \SO(31)/\SO(5)$, there  are no $\Ad(\U(13)\times \SO(5)$-invariant Einstein metrics. However,   
 for 
 $V_{24}\mathbb{R}^{31} = \SO(31)/\SO(7)$, there  are two more $\Ad(\U(12)\times \SO(7)$-invariant Einstein metrics.
 \end{remark}
In fact, it is  
  \begin{eqnarray*}
&&  F_{31, 13}(u_1)= 491774976 {u_1}^8+1682093952
   {u_1}^7+4011833808 {u_1}^6+2082493764
   {u_1}^5 \\
   &&+1342556360 {u_1}^4-2795832361
   {u_1}^3+1093464243 {u_1}^2-193555008
   {u_1}+15968016
     \end{eqnarray*}
     and
    \begin{eqnarray*}
&&  F_{31, 12}(u_1)= 24356284225
   {u_1}^8+71363530420 {u_1}^7+235478881736 {u_1}^6+125628595904 {u_1}^5
 \\ & &   +221500487082 {u_1}^4-235075487612 {u_1}^3+
   75786327156 {u_1}^2-12840182320 {u_1}+1073676289. 
     \end{eqnarray*}
   We see that $F_{31, 13}(u_1)=0$ has no real roots, but  $F_{31, 12}(u_1)=0$ has two positive roots.

Finally, we observe that the argument to show that, if the equations $ H_{n, p}(u_0) =0$ and $G_{n, p}(u_2)=0$ have  real solutions  then these are  positive, does not work for  $F_{n, p}(u_1) = 0$.  This is because  $F_{n, p}(u_1) $ has positive coefficients in odd degrees.

\end{document}